\documentclass[runningheads]{llncs}
\usepackage{graphicx}
\usepackage{amsfonts}

\def \z{\mathbb Z}

\def \sign{\rm sign}
\def \SL{\rm SL}
\def \PGL{\rm PGL}
\def \GL{\rm GL}
\def \PGL{\rm PGL}

\def \({\langle}
\def \){\rangle}

\def \isin{\rm lsin}

\def \il{\rm l\ell}

\begin{document}
\title{Continued Fraction approach to Gauss Reduction Theory
%\thanks{Supported by organization x.}
}
\author{Oleg Karpenkov}
%\orcidID{0000-0002-3358-6998} }
%
\authorrunning{O.~Karpenkov}
\institute{University of Liverpool, L69 7ZL, Liverpool, UK \\
\email{karpenk@liverpool.ac.uk}}
\maketitle              % typeset the header of the contribution
\begin{abstract}
Jordan Normal Forms serve as excellent representatives of conjugacy
classes of matrices over closed fields. Once we knows normal forms, we
can compute functions of matrices, their main invariant, etc. The situation is
much more complicated if we search for normal forms for conjugacy classes
over fields that are not closed and especially over rings.

In this paper we study $\PGL(2,\z)$-conjugacy classes of $\GL(2,\z)$ matrices.
For the ring of integers Jordan approach has various limitations and in fact it is
not effective. The normal forms of conjugacy classes of $\GL(2,\z)$ matrices
are provided by alternative theory, which is known as Gauss Reduction
Theory. We introduce a new techniques to compute reduced
forms In Gauss Reduction Theory in terms of the elements of certain
continued fractions. Current approach is based on recent progress in
geometry of numbers. The proposed technique provides an explicit
computation of periods of continued fractions for the slopes of
eigenvectors.

\end{abstract}

\section*{Introduction}

In this paper we study the structure of the conjugacy classes of
$\GL(2,\z)$. Recall that $\GL(2,\z)$ is the group of all invertible
matrices with integer coefficients. As a consequence the determinants
of such matrices are $\pm 1$.
We say that the matrices $A$
and $B$ from $\GL(2,\z)$ are {\it $\PGL(2,\z)$-conjugate} if there exists an
$\GL(2,\z)$ matrix $C$ such that $B=\pm CAC^{-1}$.
In the integer case projectivty simply means that all matrices are considered up to
the multiplication by $\pm 1$.

\vspace{2mm}

Recall that for algebraically closed fields
every matrix is conjugate to its Jordan Normal Form.
The situation with $\GL(n,\z)$ is not so simple as the set of  integer numbers does not have a field structure. 
A description of $\PGL(2,\z)$-conjugacy classes in the two-dimensional case is the subject of Gauss
Reduction Theory. The conjugacy classes are classified by periods of
certain periodic continued fractions (for additional information we refer
to~\cite{Lewis1997}, \cite{Katok2003}, and~\cite{Manin2002}). 
The first geometric invariants of $\GL(2,\z)$ matrices in the spirit of continued fractions were studied in~\cite{karpenkov-periods}. 
The questions of classification of  conjugacy classes are closely related to 
the study of homogeneous forms (see e.g. in~\cite{DeSitter})
and theory of Markov and Lagrange  spectra (see e.g. in~\cite{Cusic}).

\vspace{2mm}

Here we discuss the main elements of classical Gauss Reduction Theory based
on lattice trigonometry introduced in~\cite{itrig,itrig2} (see also in~\cite{my-book}).
Our aim is to study a natural class of reduced matrices that represent
every conjugacy class. 
It turns out that the number of reduced matrices in any $\PGL(2,\z)$-conjugacy class
of matrices is finite. 
We present a new surprising explicit formula to write 
all reduced matrices $\PGL(2,\z)$-conjugate to a given one via certain long continued fractions.
The main new method is summarised in Section~\ref{Finding reduced matrices}.
It is based on the result of  Theorem~\ref{reduction-algorithm}
which is supplemented by technical statements of Theorem~\ref{reduced-LLS}, Theorem~\ref{theorem-longcf}
and Proposition~\ref{LLS-GL}.

\vspace{2mm}

We expect that the computational complexity of the new method
is comparable to the algorithm of Chapter 7 in~\cite{my-book}.
One of the advantages of the proposed new approach is that it construct all reduced matrices 
while the classical algorithms result with a single reduced matrix.
In addition all the reduced operators of the proposed approach are explicitly described via geometric
invariants, which is potentially useful for the multidimensional case.
Recall that the studies of the conjugacy classes of $\GL(n,\z)$ for $n>2$ were motivated by V.~Arnold
(see, e.g., in~\cite{Arnold-cf}) 
who revived the notion of multidimensional continued fractions in the sense of Klein (\cite{Klein1895,Klein1896}).
The first results in higher dimensional cases were obtained in~\cite{Karpenkov-MGRT} (see also~\cite{my-book}, Chapter~21)
however the theory is far from its final form even for the case of $n=3$.
We hope that the approach of current paper will give some hints for numerous open problems in the multidimensional case.

\vspace{2mm}

This paper is organized as follows.
In Section~\ref{Background} we start with necessary notions and definitions
of geometry of numbers.
In particular  we introduce the notion of
the semigroup of reduced matrices.
We discuss three different cases of $GL(2,\z)$ matrices in general in Section~\ref{Three cases}.
In Section~\ref{Finding reduced matrices} we bring together all the stages in finding of all 
reduced matrix $\PGL(2,\z)$-conjugate to a given one.
Finally in Section~\ref{Algorithmic aspects related to computation of LLS periods}
we discuss some technical details used in the construction of reduced matrices.

\section{Background}\label{Background}

In this section we briefly discuss basic notions used in the computation of reduced matrices.
We start in Subsection~\ref{Basics of integer geometry in the plane}
with elementary notions and definitions of lattice geometry.
In Subsection~\ref{Sail and LLS sequences}
we define sails of integer angles; and introduce LLS sequences for broken lines.
Further we define LLS sequences for integer angles.
Sails and LLS sequences are important invariants related to conjugacy classes of $\GL(2,\z)$ matrices.
We continue in Subsection~\ref{LLS periods of matrices} 
with the notion of periods of LLS sequences related to matrices.
In Subsection~\ref{Matrices and continuants}
we give a continuant representation of certain class of a rather wide class of matrices (which actually 
includes all reduced matrices).
Then in Subsection~\ref{Definition of reduced matrices}
we continue with general definition of the space of reduced matrices.
We conclude this section with a general definition of difference of sequences in Subsection~\ref{Difference of sequences}.

\subsection{Basics of integer geometry in the plane}\label{Basics of integer geometry in the plane}

In this subsection we give general definitions of integer geometry.

\vspace{2mm}

We say that a point is {\it integer} if its coordinates are integers.
A segment is {\it} integer if its endpoints are integer.
An angle is called {\it integer} if its vertex is an integer point.
We also say that an integer angle is {\it rational}
if its edges contain integer points distinct to the vertex.

\vspace{2mm}

An affine transformation is said to be integer if it a one-to-one mapping of
the lattice $\z^2$ to itself. Note that the set of integer transformations is
a semidirect product of the group of translations by an integer vector and
the group $\GL(2,\z)$.

\vspace{2mm}

Two sets are {\it integer congruent}
if there exists an integer affine transformation providing a bijection
between these two sets.

\begin{definition}
The {\it integer length} of an integer segment $AB$
is the number of integer points inside its interior plus one.
Denote it by $\il(AB)$.

\vspace{1mm}

{\noindent
The {\it integer sine} of a rational angle $\angle ABC$
is defined as follows:
$$
\isin \angle ABC=
\frac{|\det(AB,BC)|}{\il(AB)\cdot\il(BC)},
$$
where $|\det(AB,BC)|$ is the absolute value of the determinant of the matrix of the pair of vectors $(AB,BC)$.
}
\end{definition}

Note that the integer lengths and integer sines are invariant under integer affine transformations.

\subsection{Sail and LLS sequences}\label{Sail and LLS sequences}

Let us now study an important invariant of angles and broken lines. It will be
employed in the proofs,
however from computational perspectives one can use the statement of Theorem~\ref{theorem-longcf} as the
explicit definition of LLS sequences for angles (without appealing to integer geometry).

\vspace{2mm}

Let $\angle ABC$ be an integer angle.
The boundary of the convex hull of all integer points in the convex closure of  $\angle ABC$ except $B$
is called the {\it sail} of $\angle ABC$.

\vspace{1mm}

Note that the sail of a rational angle is a finite broken line, while the
sail of an integer angle that is not rational is a broken line infinite to one or both sides.

\begin{definition}
Let $A_1,\ldots, A_n$ be a broken line $($here we can consider finite or infinite broken lines$)$
such that $A_i$, $A_{i+1}$ and $O$ are not in one line
for all admissible parameters of $i$,

Define
$$
\begin{array}{l}
a_{2k}=\det(OA_k,OA_{k+1}),\\
\displaystyle
a_{2k-1}=\frac{\det(A_kA_{k-1},A_kA_{k+1})}{a_{2k-2}a_k}
\end{array}
$$
for all admissible $k$.
The sequence $(a_0,\ldots, a_{2n})$ $($or an infinite one respectively$)$
is called the {\it LLS sequence} of the broken line $A_0\ldots A_n$.
\end{definition}

\begin{definition}
Consider an integer angle $\angle ABC$.
Let $\ldots A_{i-1}, A_{i}A_{i+1},\ldots$ be the sail of $\angle ABC$.  Here we consider the broken line directed from the edge 
$AB$ to the edge $BC$.
Let the LLS sequence for the broken line  $\ldots A_{i-1}, A_{i}A_{i+1},\ldots$ is
$(\ldots a_{2k-1},a_{2k},a_{2k+1},\ldots)$ $($finite or infinite$)$.
Then the sequence of absolute values
$$
(\ldots |a_{2k-1}|,|a_{2k}|,|a_{2k+1}|,\ldots)
$$
is called the LLS sequence of the angle $\angle ABC$ and denoted by $LLS(\angle ABC)$,
\end{definition}

\begin{remark}
Notice that if we consider rational angle $\angle ABC$ with a positive value of $\det(AO, BC)$
then its LLS sequence $(a_0,\ldots, a_{2n})$ consists of odd number of elements and
$$
\begin{array}{l}
a_{2k}=\il A_kA_{k+1},\\
a_{2k-1}=\isin \angle A_{k-1}A_kA_{k+1}
\end{array}
$$
for all admissible $k$. This explains the abbreviation LLS (which is Lattice Length-Sine) sequence.
\end{remark}

Let us formulate the  following important geometric property of  LLS sequences.

\begin{theorem}{\bf (\cite{itrig} 2008)}
Consider a finite broken line $A_1,\ldots, A_n$ with the LLS sequence $(a_0,\ldots, a_{2n})$.
Let also $A_0=(1,0) $ and $A_1=(1,a_0)$.
Then
$$
A_n=\big(K_{2n+1}(a_0,\ldots, a_{2n}),K_{2n}(a_1,\ldots, a_{2n})\big).
$$
\end{theorem}

For an additional information on continued fractions and related integer geometry related we refer an interested reader to
the monograph~\cite{my-book}.

\subsection{LLS periods of $\GL(2,\z)$ matrices}\label{LLS periods of matrices}

Further let us show how to relate matrices $M$ with finite sequences of positive integers.

\vspace{2mm}

Let $M$ be a $(2\times 2)$-matrix with two distinct real eigenvalues.
In this case $M$ has two eigenlines. The complement to these eigenlines 
is a union of four cones. 
We say that the sails of these cones are the {\it sails associated to $M$}.

\begin{definition}
We say that a sequence of positive integers  is an {\it LLS sequence of a matrix} $M$,
if this sequence is the LLS sequence of one of the sails associated to $M$. 
\end{definition}

\begin{remark}
It turns out that in the case of $\GL(2,\z)$ matrices with real irrational eigenvalues
the LLS sequences of all associated sails coincide up to a possible index shift
(see Section~7 of~\cite{my-book}).
So the LLS sequence is uniquely defined by the matrix in this case. 

\end{remark}

We conclude this subsection with the following fundamental definition. 

\begin{definition} 
Let $M$ be a $\GL(2,\z)$ matrix with real irrational eigenvalues
then its LLS sequence is periodic.  In addition the matrix $M^2$ is acting as a periodic shift
on every of the sails. Assume that $M^2$ shifts the sail by $n$ vertices.
Then any period of length $n$ is called an {\it LLS period of $M$}.
$($Here we write the elements of the period in the order from a vertex $v$ on the sail to the vertex $M^2(v)$ on the sail.$)$
\end{definition}

\begin{remark}
Note that inverse matrices to each other have reversed periods. 
\end{remark}

\subsection{Matrices and continuants}\label{Matrices and continuants}

In this section we show that for certain class of $\GL(n,\z)$ matrices
their elements have a nice representation in terms of continuants.

\vspace{1mm}

Recall first the definition of the continuant.

\begin{definition}
Let $n$ be a positive integer.
A {\it continuant} $K_n$ is a polynomial with integer coefficients defined recursively by
$$
\begin{array}{l}
K_{-1}()=0;\\
K_0()=1;\\
K_1(a_1)=a_1;\\
K_n(a_1,a_2,\ldots,a_n)=a_nK_{n-1}(a_1,a_2,\ldots,a_{n-1})+K_{n-2}(a_1,a_2,\ldots,a_{n-2}).
\end{array}
$$
\index{continuant}
\end{definition}

\begin{remark}\label{continuant-remark}
Note that
$$
[a_1;a_2:\cdots:a_n]=\frac{K_n(a_1,a_2,\ldots,a_n)}{K_{n-1}(a_2,a_2,\ldots,a_n)}.
$$
\end{remark}

Secondly we fix the following notation.

\begin{definition}\label{Def-Ma}
Let $a$ be a real number, denote by $M_a$ the following matrix:
$$
M_a=\left(
\begin{array}{cc}
0&1\\
1&a\\
\end{array}
\right).
$$
Now let $(a_1,\ldots,a_n)$ be any sequence of real numbers, we set
$$
M_{a_1,\ldots,a_n}=
\prod\limits_{k=1}^{n}
\left(
\begin{array}{cc}
0&1\\
1&a_k\\
\end{array}
\right).
$$
\end{definition}

Finally let us show that the above matrices have the following simple explicit expression for their coefficients.
We will use it later in the construction of reduced matrices.

\begin{proposition}\label{proposition-continuants-general}
Let $n\ge 0$ and let $(a_1,\ldots,a_n)$ be any sequence of real numbers.
Then we have
$$
M_{a_1,\ldots,a_n}=
\left(
\begin{array}{cc}
K_{n-2}(a_2,\ldots,a_{n-2})&K_{n-1}(a_2,\ldots,a_{n})\\
K_{n-1}(a_1,a_2,\ldots,a_{n-1})&K_{n}(a_1,a_2,\ldots,a_{n})\\
\end{array}
\right).
$$
In addition, we have
$$
\det M = (-1)^n.
$$
\end{proposition}

\begin{example}
Consider
$$
M_{3,-3,-2,5}=
M_{3}\cdot M_{-3}\cdot M_{-2}\cdot M_{5}.
$$
Hence $M$ is represented by the following sequence:
$(3,-3,-2,5)$.
By Proposition~\ref{proposition-continuants-general}
we immediately have
$$
M_{3,-3,-2,5}=
\left(
\begin{array}{cc}
K_{2}(-3,-2)&K_{3}(-3,-2,5)\\
K_{3}(3,-3,-2)&K_{4}(3,-3,-2,5)\\
\end{array}
\right).
$$
Therefore,
$$
M_{3,-3,-2,5}=\left(
\begin{array}{cc}
7&32\\
19&87\\
\end{array}
\right).
$$
Here we actually have
$$
\frac{19}{7}=[3:-3:-2]
\quad \hbox{and} \quad
\frac{87}{19}=[3:-3:-2:5].
$$

Note also that
$$
\det  M =(-1)^4=1.
$$
\end{example}

{\noindent
{\it Proof of Proposition~\ref{proposition-continuants-general}.}}
The proof is done by induction in $n$.

\vspace{2mm}

{\noindent
{\it Base of induction.}
For $n=1$ we have
$$
M_{a_1}
=\left(
\begin{array}{cc}
0&1\\
1&a_1\\
\end{array}
\right)
=
\left(
\begin{array}{cc}
K_{-1}()&K_{0}()\\
K_{0}()&K_{1}(a_1)\\
\end{array}
\right).
$$
For $n=2$ we have
$$
M_{a_1,a_2}=M_{a_1}M_{a_2}=
\left(
\begin{array}{cc}
1&a_2\\
a_1&1+a_1a_2\\
\end{array}
\right)
=\left(
\begin{array}{cc}
K_{0}()&K_{1}(a_2)\\
K_{1}(a_1)&K_{2}(a_1,a_2)\\
\end{array}
\right).
$$
}

\vspace{2mm}

{\noindent
{\it Step of induction.}
We have
$$
\begin{array}{ll}
& M_{a_1,\ldots,a_{n+1}}=M_{a_1,\ldots,a_{n}}\cdot M_{a_{n+1}}=
\\
& \qquad
\left(
\begin{array}{cc}
K_{n-2}(a_1,\ldots,a_{n-2})&K_{n-1}(a_2,\ldots,a_{n})\\
K_{n-1}(a_1,\ldots,a_{n-1})&K_{n}(a_1,\ldots,a_{n})\\
\end{array}
\right)
\cdot
\left(
\begin{array}{cc}
0&1\\
1&a_{n+1}\\
\end{array}
\right)
=
\\
&
\qquad
\left(
\begin{array}{cc}
K_{n-1}(a_2,\ldots,a_{n})& K_{n-2}(a_2,\ldots,a_{n-2})+a_{n+1} K_{n-1}(a_2,\ldots,a_{n})\\
K_{n}(a_1,\ldots,a_{n})& K_{n-1}(a_1,\ldots,a_{n-1})+a_{n+1}K_{n}(a_1,\ldots,a_{n})\\
\end{array}
\right)
=
\\&
\qquad
\left(
\begin{array}{cc}
K_{n-1}(a_2,\ldots,a_{n})&K_{n}(a_2,\ldots,a_{n+1})\\
K_{n}(a_1,\ldots,a_{n})&K_{n+1}(a_1,\ldots,a_{n+1})\\
\end{array}
\right).
\end{array}
$$
}
The last inequality is a classical relation for the numerators and denominators of continued fractions
(see, e.g., in~\cite{Khinchin-book} or in~\cite{my-book}).
%PROPOSITION~1.13.
This concludes the proof for the induction step.

\vspace{2mm}

Finally, since $\det M_a=-1$ we have
$$
\det M=(-1)^n.
$$
\qed

\subsection{Definition of reduced matrices}\label{Definition of reduced matrices}

In general there are several ways to set up reduced matrices.
Here we describe one of them.
There are two main benefits for the proposed choice of reduced matrices.
Firstly, they form a semigroup with respect to the matrix multiplication.
Secondly, there is a simple description of such matrices 
in terms of continuants (see Proposition~\ref{proposition-continuants-general}).

\begin{definition}
Consider a sequence of positive integers $(a_1,\ldots.a_n)$.
Then the matrix $M_{a_1,\ldots,a_n}$ is said to be {\it reduced}.
\index{matrix!reduced}
\end{definition}

Directly from the definition of reduced matrices we have the following remarkable property.

\begin{proposition}
The set of all reduced matrices is a semigroup with respect to matrix multiplication.
\qed
\end{proposition}

\subsection{Difference of sequences}\label{Difference of sequences}

Finally let us give the following general combinatorial definition.

\begin{definition}
Let $m>n$ be two non-negative integers and
consider two sequences of real numbers
$$
S_a=(a_1,\ldots, a_m) \quad \hbox{and} \quad S_b=(b_1,\ldots, b_n).
$$
We say that there exists a {\it difference} of $S_a$ and $S_b$ if
There exists $k\le m+1$ such that the following conditions are fulfilled
\begin{itemize}
\item $b_i = a_i$ for $1\le i<k$;
\item either $k=m+1$ or $b_k\ne a_k$;
\item $b_{k+i}=a_{k+i+m-n}$ for $0\le i\le n-k$.
\end{itemize}
In this case we denote
$$
S_a-S_b=(a_k,a_{k+1},\ldots, a_{k+n-m-1}).
$$
\end{definition}

\begin{example}
We have
$$
(1,2,3,4,5,6,7,8)-(1,2,3,6,7,8)=(4,5).
$$
\end{example}

\begin{example}
The expression
$$
(1,2,3,4,5,6,7,8)-(1,4,8).
$$
is not defined.
\end{example}

\section{Three cases of $GL(2,\z)$ matrices}\label{Three cases}
It is natural to split the matrices of  $\GL(2,\z)$ into three cases with respect to their spectra (set of eigenvalues).
We distinguish the cases of complex, rational, and real irrational spectra.
The cases of complex and rational cases are rather straightforward, they are not included to Gauss Reduction Theory. 
The case of real irrational spectra is more complicated, it is central for this paper.

\vspace{1mm}

Let us now briefly discuss these three cases in this section.

\vspace{2mm}

{\noindent
{\bf Case of complex spectra}: We start with $\GL(2,\z)$ matrices whose characteristic polynomials
have a pair of complex conjugate roots.
There are exactly three
$\PGL(2,\z)$-conjugacy classes of such matrices; they are represented by
$$
\left(
\begin{array}{cc}
1&1\\
-1&0\\
\end{array}
\right), \quad
 \left(
\begin{array}{cc}
0&1\\
-1&0\\
\end{array}
\right), \quad
\hbox{and} \quad
\left(
\begin{array}{cc}
0&1\\
-1&-1\\
\end{array}
\right).
$$
}
These classes are perfectly distinguished by traces of matrices.

\vspace{2mm}

{\noindent
{\bf Case of rational spectra}: It turns out that such matrices have eigenvalues equal to $\pm 1$, 
any of rational spectra matrices is $\PGL(2,\z)$-conjugate to exactly one of the following matrices
$$
\left(
\begin{array}{cc}
1&m\\
0&1\\
\end{array}
\right)
\quad \hbox{for} \quad m\ge 0,
\qquad
\left(
\begin{array}{cc}
1&0\\
0&-1\\
\end{array}
\right),
\qquad
\hbox{or}
\qquad
\left(
\begin{array}{cc}
1&1\\
0&-1\\
\end{array}
\right).
$$
(Note that the rational spectra case contains degenerate case of two coinciding roots.
Indeed a double root of a quadratic polynomial with integer coefficients is always a rational number.)
}

\vspace{2mm}

{\noindent
{\bf Case of real irrational spectra}: this case is the most complicated.
It is described by a so-called Gauss Reduction Theory, which 
is based on Euclidean types algorithms that provide a descend to reduced matrices
(see e.g. in Chapter 7 of~\cite{my-book}).
It is interesting to note that the number of reduced matrices integer congruent to a given one is
finite and equal to the number of elements in the minimal period
of the regular continued fraction for the tangent of the slope of any eigenvector of the matrix.
In the next section we introduce an alternative algorithm based on explicit expressions for 
reduced matrices originated in geometry of numbers.
}

\vspace{2mm}

\section{Techniques to find reduced matrices $\PGL(2,\z)$-conjugate to a given one}\label{Finding reduced matrices}

Let us outline the main stages of the reduced matrices construction. All the statements involved in it are proven
in the next section.
The construction is based on general Theorem~\ref{reduction-algorithm}
and several supplementary technical statements.

\vspace{2mm}

{\noindent
{\bf  Input data.}
We are given a $\GL(2,\z)$ matrix. Namely we have 
$$
M=\left(
\begin{array}{cc}
p&r\\
q&s\\
\end{array}
\right).
$$
}

\vspace{2mm}

{\noindent {\bf Goal of the algorithm.} List all reduced matrices $\PGL(2,\z)$-conjugate to $M$. }

\vspace{2mm}

{
\noindent
{\bf Step 1.}
Starting with any point $P_0$
set
$$
\quad P_1=M^4(P_0) \quad \hbox{and} \quad P_2=M^6(P_0).
$$
and compute $LLS(\angle P_0OP_1)$ and  $LLS(\angle P_0OP_2)$ using Theorem~\ref{theorem-longcf}.
}

\vspace{2mm}

{
\noindent
{\bf Step 2.}
By Proposition~\ref{LLS-GL} one of the periods of LLS sequence for $M$ is a half of
$$
LLS(\angle P_0OP_2)-LLS(\angle P_0OP_1).
$$
We take the first half of this sequence, so let the period be 
$$
(a_1,\ldots,a_n)
$$
and let the lengths of minimal possible periods  be $m$.
}

\vspace{2mm}

{
\noindent
{\bf Step 3.}
Now we can write down the reduced matrices in accordance with Theorem~\ref{reduced-LLS}
and Proposition~\ref{proposition-continuants-general}.
}
\vspace{2mm}

{\noindent
{\bf Output.}
All the reduced matrices $\PGL(2,\z)$-conjugate to $M$ will be of the form
$$
\left(
\begin{array}{cc}
K_{n-2}(a_{k+2},\ldots,a_{k+n-2})&K_{n-1}(a_{k+2},\ldots,a_{k+n})\\
K_{n-1}(a_{k+1},a_{k+2},\ldots,a_{k+n-1})&K_{k+n}(a_{k+1},a_{k+2},\ldots,a_{k+n})\\
\end{array}
\right),
$$
here $k=0,\ldots,m-1$.
}

\begin{example}
{\bf Input:} Let us find all reduced matrices for the matrix
$$
M=\left(
\begin{array}{cc}
7&-30\\
-10&43\\
\end{array}
\right).
$$ 
{
\noindent
{\bf Step 1.}
Starting with any point $P_0=(1,1)$
set
$$
\begin{array}{l}
 P_1=M^4(P_0)=(-2875199,4119201) \quad \hbox{and} \\
 P_2=M^6(P_0)=(-7182245951, 10289762449).
 \end{array}
$$
Let us first compute $LLS(\angle P_0OP_1)$.
First of all note that
$$
\varepsilon=-\sign\frac{1}{1}=-1
\quad \hbox{and}
\quad
\delta=\frac{-2875199}{4119201}=-1.
$$
In addition
$$
\det(OP_1OP_2)\cdot (-1)>0.
$$
Therefore, we consider the following odd regular continued fractions
$$
\begin{array}{l}
\frac{1}{1}=[1];\\
\left|\frac{-2875199}{4119201}\right|=\frac{2875199}{4119201}=
[0; 1:2:3:4:1:2:3:4:1:2:3:4:1:2:3:3].
\end{array}
$$
No we combine these two continued fractions in accordance with Theorem~\ref{theorem-longcf}:
$$
\begin{array}{l}
[-1; 0: 0: -1: -2: -3: -4: -1: -2: -3: -4: -1: -2:
 -3: -4: \\
 \qquad \qquad \qquad \qquad \qquad
 \displaystyle -1: -2: -3: -3]=\frac{-6994400}{4119201}.\\
 \end{array}
$$
We have
$$
\begin{array}{l}
 \displaystyle
\left|\frac{-6994400}{4119201}\right|=\frac{6994400}{4119201}
\\
 \qquad \qquad \quad
 =[1; 1: 2: 3: 4: 1: 2: 3: 4: 1: 2: 3: 4: 1: 2: 3: 3].
\end{array}
$$
Therefore,
$$
LLS(\angle P_0OP_1)=(\underline{1, 1, 2, 3, 4, 1, 2, 3, 4, 1, 2, 3, 4, 1, 2, 3}, \overline{3})
$$
Similarly we get
$$
LLS(\angle P_0OP_2)=(\underline{1, 1, 2, 3, 4, 1, 2, 3, 4, 1, 2, 3, 4, 1, 2, 3}, \framebox{4, 1, 2, 3, 4, 1, 2, 3}, \overline{3}).
$$
}
(Here we show the difference of the sequences in the box.)

\vspace{2mm}

{
\noindent
{\bf Step 2.}
By Proposition~\ref{LLS-GL} one of the periods of the LLS sequence for $M$ is 
a half of the sequence
$$
LLS(\angle P_0OP_2)-LLS(\angle P_0OP_1)=(4,1,2,3,4,1,2,3),
$$
which is 
$$(4,1,2,3).$$
Note that the minimal possible period is pf length $4$.
}

\vspace{2mm}

{
\noindent
{\bf Step 3.}
We can write down the reduced matrices in accordance with Theorem~\ref{reduced-LLS}
and Proposition~\ref{proposition-continuants-general}
for all distinct periods of length 4, i.e., for
$$
(4,1,2,3), \quad
(1,2,3,4), \quad
(2,3,4,1), \quad \hbox{and} \quad
(3,4,1,2).
$$
}

\vspace{2mm}

{\noindent
{\bf Output.}
Finally applying Proposition~\ref{proposition-continuants-general} to these four sequences we have
the list of all reduced matrices $\PGL(2,\z)$-conjugate to $M$:
$$
\left(
\begin{array}{cc}
K_{2}(1,2)&K_{3}(1,2,3)\\
K_{3}(4,1,2)&K_{4}(4,1,2,3)\\
\end{array}
\right)=
\left(
\begin{array}{cc}
3&10\\
14&47\\
\end{array}
\right),
\quad
\left(
\begin{array}{cc}
7&30\\
10&43\\
\end{array}
\right),
\quad
\left(
\begin{array}{cc}
13&16\\
30&37\\
\end{array}
\right),
\quad
\left(
\begin{array}{cc}
5&14\\
16&45\\
\end{array}
\right).
$$
}
(We show continuants only for the first matrix and omit them for the others.)
\end{example}

\section{Technical aspects related to computation of reduced matrices}\label{Algorithmic aspects related to computation of LLS periods}

In this section we prove some technical statements involved in justification of the above algorithm. 
We start in Subsection~\ref{Continued fraction enumeration of reduced matrices}
with writing periods of LLS sequences for reduced matrices.
In Subsection~\ref{Matrices PGL-conjugate to a given one}
we explain how to list all reduced matrices $\PGL(2,\z)$-conjugate to the given one (the reduced matrices are given in terms of
LLS periods of original matrices).
Then we show in general  how to compute LLS sequences of angles in Subsection~\ref{Computation of LLS sequences for rational angles}.
Finally in Subsection~\ref{Periods of the LLS sequences corresponding to matrices}
we give the algorithm for computation of LLS sequence periods.

\subsection{Continued fraction enumeration of reduced matrices}\label{Continued fraction enumeration of reduced matrices}

Let us find a period of the LLS sequence for matrices $M_{a_1,a_2,\ldots ,a_{n}}$.

\begin{theorem}\label{reduced-LLS}
Let  $n$, $a_1,\ldots, a_n$  be positive integers.
Then one of the periods of the LLS sequence for $M_{a_1,a_2,\ldots ,a_{n}}$ is
$$
(a_1,a_2,\ldots ,a_{n}).
$$
\end{theorem}

\begin{proof}
Consider the sequence of integer points
$$
(x_k,y_k)=M_{a_1,a_2,\ldots ,a_{n}}^k(1,0), \quad \hbox{for $k=1,2,3,\ldots$}
$$
By Item~{\it $($i$)$} for every $k$ the coordinates $x_k$ and $y_k$ are relatively prime and
$$
\frac{y_k}{x_k}=[(a_1;a_2:\cdots:a_n)^k].
$$
Therefore, all the points $(x_k,y_k)$
 are vertices of  the sail the periodic continued fraction
$$
\alpha=[(a_1;a_2:\cdots:a_n)].
$$
(This is a classical statement of geometry of numbers (Theorem~3.1 of~\cite{my-book}).)
This immediately implies that the direction of the vector $(1,\alpha)$ is
the limiting direction for the sequence of directions for the vectors $(x_k,y_k)$,
and in particular that
$$
\lim\limits_{k\to \infty}\frac{y_k}{x_k}=\alpha.
$$
Hence $(1,\alpha)$ is one of the eigenvectors corresponding to
the maximal eigenvalue (and thus  the eigenvalues are both real and distinct).

By construction the LLS sequence for $\alpha$ is periodic with period
$$
(a_1,a_2,\ldots,a_n).
$$

Finally 
%by PROPOSITION~8.13
the sail for $\alpha$ from some element coincides with the sail for $M$.
Since the sail for $M$ is periodic, the period is the same as for $\alpha$,
i.e.,
$$
(a_1,a_2,\ldots,a_n).
$$
This concludes the proof.
\qed
\end{proof}

\subsection{Matrices $\PGL(2,\z)$-conjugate to a given one}\label{Matrices PGL-conjugate to a given one}

The following theorem produces the list of all reduced matrices $\PGL(2,\z)$-conjugate to a given one.

\begin{theorem}\label{reduction-algorithm}
Let $M$ be a $\GL(2,\z)$ matrix and let
$$
(a_1,\ldots,a_n)
$$
be a period of LLS sequence corresponding to $M$.
Finally let $m$ be the minimal lengths of the period of the LLS sequence.
Then the list of all reduced matrices $\PGL(2,\z)$-conjugate to $M$
consists of $m$ matrices of the form
$$
M_{a_{1+k},\ldots,a_{n+k}} \quad \hbox{for $k=1,\ldots, m$}.
$$
\end{theorem}

\begin{proof}
We know that two operators have the same LLS sequences if and only if their unions of eigenlines are integer
congruent to each other.
Hence $M$ could be congruent only to reduced matrices commuting with
$$
\pm M_{a_{1+k},\ldots,a_{n+k}} \quad \hbox{for $k=1,\ldots, m$}.
$$
(These are the only matrices that have such LLS sequences.) Such matrices are some powers of these matrices.

Finally the shift of the LLS sequence of $M^2$ by $n$ vertices uniquely determines the reduced matrices.
\qed
\end{proof}

\subsection{Computation of LLS sequences for rational angles}\label{Computation of LLS sequences for rational angles}

In this subsection we formulate a theorem that provides
an explicit techniques to write the LLS sequence directly from the values of the elements of
a given matrix. We start with the following remark.

\begin{remark}\label{removing-0a_1-remark}
Recall one technical statement for angles represented by slopes with tangents less than 1:
the angles represented by the continued fractions
$$
[0;a_1:a_2:\cdots:a_{2n}]
\quad \hbox{and} \quad
[a_2;\cdots:a_{2n}]
$$
are integer congruent. In particular, they have the same LLS sequences.
\end{remark}

Now we can formulate the following result. 

\begin{theorem}\label{theorem-longcf}
Consider two linearly independent integer vectors
$$
A=(p,q) \quad \hbox{and} \quad  B=(r,s).
$$
We assume that none of them are proportional either to $(1,0)$ or to $(0,1).$
Let two sequences of integers
$$
(a_0,a_1,\ldots ,a_{2m}) \quad \hbox{and} \quad (b_0,a_1,\ldots ,b_{2n})
$$
be defined as the sequences of elements of the odd regular continued fractions of

\begin{itemize}
\item  integers $|q/p|$ and $|s/r|$ in case of $\det (OA,OB)\cdot \sign\frac{p}{q}<0$;

\item  integer $|p/q|$ and $|r/s|$ in case of $\det (OA,OB)\cdot \sign\frac{p}{q}>0$.
\end{itemize}
Further we set
$$
\varepsilon=-\sign\frac{p}{q}
\quad \hbox{and} \quad
\delta=\sign\frac{r}{s}.
$$
Denote
$$
\alpha=[\varepsilon a_{2m}: \varepsilon a_{2m-1}:\cdots:\varepsilon a_1:\varepsilon a_{0}:0:\delta b_0:
\delta b_1:\cdots: \delta b_{2n}].
$$
Let
$$
|\alpha|=[c_0;c_1: \cdots: c_{2k}]
$$
be the regular odd continued fraction for $|\alpha|$.
Set
\begin{itemize}
\item $S=(c_0,c_1,\ldots, c_{2k})$ in case if $c_0\ne 0$;
\item $S=(c_2,\ldots, c_{2k})$ in case if $c_0= 0$.
\end{itemize}
Then $S$ is the LLS sequence for the  angle $\angle AOB$.
\end{theorem}

\begin{remark}
In fact it is possible to simplify the computation of the continued fraction for
$$
\alpha=[\varepsilon a_{2m}: \varepsilon a_{2m-1}:\cdots:\varepsilon a_1:\varepsilon a_{0}:0:\delta b_0:
\delta b_1:\cdots: \delta b_{2n}],
$$
namely we do not need to find $(\delta b_1:\cdots: \delta b_{2n})$.
\\
In the case of  $\det (OA,OB)\cdot \sign\frac{p}{q}<0$ we can simply take
$$
\alpha=[\varepsilon a_{2m}: \varepsilon a_{2m-1}:\cdots:\varepsilon a_1:\varepsilon a_{0}:0: q/p];
$$
in the case of  $\det (OA,OB)\cdot \sign\frac{p}{q}>0$ we have
$$
\alpha=[\varepsilon a_{2m}: \varepsilon a_{2m-1}:\cdots:\varepsilon a_1:\varepsilon a_{0}:0: p/q].
$$
\end{remark}

%\begin{remark}
%For the completeness of the observation let us consider the cases of vectors $(1,0)$
%and $(0,1)$.
%\end{remark}

\begin{example}
Consider the angle $\alpha=\angle AOB$ with
$$
A=(8,2) \quad \hbox{and} \quad B=(6,21).
$$
Let us compute its LLS sequence using the techniques suggested by Theorem~\ref{theorem-longcf}.

Note first that
$$
\frac{8}{2}\cdot \det \left(
\begin{array}{cc}
8&6\\
2&21\\
\end{array}
\right)=4\cdot 156=624 >0,
$$
hence we consider $8/2$ and $6/21$ respectively.
We have
$$
\varepsilon=-\sign\frac{8}{2}=-1
\quad \hbox{and} \quad
\delta=\sign\frac{6}{21}=1.
$$
Further we have
$$
\frac{8}{2}=4=[4] ,\quad \hbox{and} \quad \frac{6}{21}=\frac{2}{7}=[0;3:2].
$$
So the expression for the long continued fraction is as follows:
$$
[-4;0:0:3:2]=-\frac{26}{7}.
$$
Let us now write the odd continued fraction for $|-26/7|$:
$$
\left|-\frac{26}{7}\right|=[3;1:2:1:1].
$$
Since the first element of the continued fraction is not equal to zero ($3\ne 0$)
the LLS sequence for $\alpha$ is
$$
(3,1,2,1,1).
$$
\end{example}

\noindent{{\it Proof of Theorem~\ref{theorem-longcf}.}
First we set $E=(1,0)$.
Consider the broken line that is a concatenation of the sail of the angle $\angle AOE$ (in case if the last edge of this sail is
not vertical we add the infinitesimal edge $EE$ of zero integer length with vertical direction and $0$ integer length) and
the sail for the angle $\angle EOB$ (again we add another infinitesimal edge $EE$ in case
if the first edge of the sail of the angle is not vertical).

Note that this broken line $L$ have the following properties:

--- it starts at the ray $OA$ and ends at the ray $OB$;

--- the direction of the first edge is towards the interior of the angle $\angle AOB$.

Then the angle is integer congruent to the angle $\angle EOC$ with
$C=(1,\alpha)$ where
$|\alpha|$ is defined by the LLS sequence of the above broken line as
$$
\alpha=[ \varepsilon a_{2m}: \varepsilon a_{2m-1}:\cdots:\varepsilon a_1:\varepsilon a_{0}:0:\delta b_0:
\delta b_1:\cdots: \delta b_{2n}].
$$
The proof for this formula is given by the study numerous straightforward cases of various signs for $p,q,r,s$ and $\det (OA,OB)$.

\vspace{2mm}

Let us study the case $p,q,r,s>0, \det(OA,OB)<0$.

In this case, the first part of the broken line $L$ will be the sail of $\angle AOE$ passed clockwise.
Hence the elements of the LLS sequence will be reversed and negative to the values of the LLS sequence for $\angle AOE$.
Note that in case if $q/p<1$ we end up with an infinitesimal (zero integer length) vertical vector
which additionally brings two elements: 
the element $\lfloor p/q \rfloor$ for the angle with the vertical line passing through $E$,
and the element $0$ indicating that we stay at $E$,.
Then we switch to the second sail. Both sails are starting vertically
(or asymptotically vertical in case if $a_1$ or $b_1$ are zeroes), hence the angle between the edges
corresponding to $a_0$ and $b_0$ is zero. So we add a zero element to the LLS sequence for $L$ here.
Finally we continue back following the sail of the angle $\angle  EOB$, which is described by the continued fraction
$$
[b_0:b_1:\cdots: b_{2n}]
$$
(here again we have $b_0=0$ and $b_1=\lfloor s/r\rfloor $ for the case of $r/s<1$).
Hence the LLS sequence of the broken line $L$ is
$$
(- a_{2m}, -a_{2m-1},\ldots,-a_1, -a_{0},0, b_0,
b_1,\ldots, b_{2n}).
$$
Finally we get
$$
\alpha=
[- a_{2m}: -a_{2m-1}:\cdots:-a_1: -a_{0}:0: b_0:
b_1:\cdots: b_{2n}].
$$

The cases for the rest choices of signs  for $p,q,r,s$ and $\det (OA,OB)$ are considered similarly,
so we omit them here.

\vspace{2mm}

Let now
$$
|\alpha|=[c_0;c_1: \cdots: c_{2k}].
$$
Therefore  (c.f. Remark~\ref{removing-0a_1-remark},) the LLS sequence for $\angle EOC$ is either 
$$
( c_0,c_1,c_2, \ldots, c_{2k}) \quad \hbox{ in case if $c_0\ne 0$;}
$$
or
$$
(c_2, \ldots, c_{2k}) \quad \hbox{ in case if $c_0\ne 0$}
.
$$
\qed

\subsection{Periods of the LLS sequences corresponding to matrices}\label{Periods of the LLS sequences corresponding to matrices}

In this subsection we show how to extract periods of the LLS sequence for a given matrix.

\begin{proposition}\label{LLS-GL}
Let a $\GL(2,\z)$ matrix $M$ has distinct irrational eigenvalues $($not necessarily positive$)$.
Let also $P_0$ be any non-zero integer point.
Denote
$$
P_1=M^4(P_0), \quad \hbox{and} \quad P_2=M^6(P_0).
$$
Then there exists a difference
$$
LLS(\angle P_0OP_2) - LLS(\angle P_0OP_1),
$$
which is a period of the LLS sequence for $M$ repeated twice.
\end{proposition}

\begin{remark}
The obtained period of the LLS sequence might be not of the minimal lengths.
\end{remark}

We start the proof with the following lemma.

\begin{lemma}\label{M3-M2}
Let a $\GL(2,\z)$ matrix $M$ has distinct irrational positive eigenvalues.
Let also $P_0$ be any non-zero integer point.
Denote
$$
P_1=M^2(P_0), \quad \hbox{and} \quad P_2=M^3(P_0).
$$
Then there exists a difference
$$
LLS(\angle P_0OP_2) - LLS(\angle P_0OP_1),
$$
which is a period of the LLS sequence for $M$.
\end{lemma}

\begin{proof}
Set $Q=M(P_0)$.
First of all note that $\angle P_0OQ$ is a fundamental domain of one of the angles $C$ whose edges are
eigenvectors of $M$
up to the action of the group of (integer) powers of $M$.
Hence it contains at last one vertex of the sail.
Denote this vertex by $v$.
Then the angle $\angle P_0OP_2$ contains vertices
$$
v_0=v, \quad v_1=M(v),  \quad \hbox{and} \quad \quad v_2=M^2(v).
$$
Thus by convexity reasons, the sail for the angle $\angle P_0OP_2$
contains the part of the sail of $C$ between $v_0$ and $v_2$.

Namely there will be four parts of the sail:
\begin{itemize}
\item $S_1$: a part of the sail contained in $P_0Ov_0$;
\item $S_2$: a part of the sail contained in $v_0Ov_1$;
\item $S_3$: a part of the sail contained in $v_1Ov_2$;
\item $S_4$: a part of the sail contained in $v_2OP_2$.
\end{itemize}
Here $S_2$ and $S_3$ are periods of the sail for the angle $\angle P_0OP_2$.

\vspace{2mm}

Now by the same reason we have $v_0$ and $v_1$ in the sail for  angle $\angle P_0OP_1$.
Here we have the following parts
\begin{itemize}
\item $S'_1$: a part of the sail contained in $P_0Ov_0$;
\item $S'_2$: a part of the sail contained in $v_0Ov_1$;
\item $S'_3$: a part of the sail contained in $v_1OP_1$.
\end{itemize}

\vspace{2mm}

Note that
$$
\left\{
\begin{array}{l}
S'_1=S_1;\\
S'_2=S_2;\\
S'_2\cong S_3;\\
S'_3\cong S_4.
\end{array}
\right.
$$
Therefore, the difference of the LLS sequences for the angle $\angle P_0OP_2$ and the angle $\angle P_0OP_1$
is precisely the period of the LLS sequence between the points $v_1$ and $v_2$.
This period correspond to $M$ as $M(v_1)=v_2$.
This concludes the proof.
\qed
\end{proof}

\begin{remark}
It is not enough to consider the difference of the LLS sequences
for the angles $\angle P_0OP_1$ and $\angle P_0OQ$ (where $Q=M(P_0)$), as it is not possible to determine the last integer sine of the period then.
Let us illustrate this with the following example.

Consider a matrix
$$
M=
\left(
\begin{array}{cc}
1&2\\
1&3\\
\end{array}
\right).
$$
Let us aslo take the point $P=(4,-1)$.
Then
$$
Q=M(P_0)=(2,1), \quad P_1=M^2(P_0)=(4,5), \quad \hbox{and} \quad P_2=M^3(P_0)=(14,19).
$$
The LLS sequences for the angles $\angle P_0OQ$, $\angle P_0OP_1$
and $\angle P_0OP_2$ are respectively
$$
\begin{array}{l}
(1,4,1);\\
(1,3,1,3,1);\\
(1,3,1,2,1,3,1).\\
\end{array}
$$
We have
$$
(1,3,1,2,1,3,1)-(1,3,1,3,1)=(2,1)
$$
which is the correct period for the LLS sequence of $M$,
while the difference
$$
(1,3,1,3,1)-(1,4,1)
$$
is not even defined.
\end{remark}

{\noindent {\it Proof of Proposition~\ref{LLS-GL}.}
First of all let us study the LLS sequences of reduced operators.
Let
$$
M=M_{a_1,\ldots,a_n}
$$
be a reduced operator for the sequence of positive integers $(a_1,\ldots, a_n)$.
Then from Definition~\ref{Def-Ma} we have
$$
M^2=M^2_{a_1,\ldots,a_n}=M_{a_1,\ldots,a_n,a_1,\ldots,a_n}.
$$
Hence the period of the LLS sequence corresponding to $M^2$ is twice the period of $M$.
}

\vspace{2mm}

For an arbitrary $M$ we know that 
$$
M^2\cong M_{a_1,\ldots,a_n,a_1,\ldots,a_n}=M^2_{a_1,\ldots,a_n}.
$$
Hence $M$ itself is $\PGL(2,\z)$-congruent to $M_{a_1,\ldots,a_n}$  
Therefore,  
the period of LLS sequence corresponding to $M^2$ will be twice the period of
the LLS sequence for $M$.

\vspace{2mm}

By Lemma~\ref{M3-M2} the difference
$$
LLS(\angle P_0OP_3) - LLS(\angle P_0OP_2)
$$
exists and it is a period for $M^2$.
Finally by the above the resulting sequence is a period of the LLS sequence for $M$ repeated twice.
\qed


\end{document}